\documentclass{daj}
\usepackage{amsmath,amsthm,amsfonts,amsrefs}

\newcommand{\field}[1]{\mathbb{#1}}

\newcommand{\Z}{\field{Z}}

\newcommand{\F}{\field{F}}

\newcommand{\tensor} {\otimes}

\newcommand{\MM}{\mathcal{M}}

\newcommand{\set}[1]{\{#1\}}

\newcommand{\beq}{\begin{displaymath}}
\newcommand{\eeq}{\end{displaymath}}
\newcommand{\beqn}{\begin{equation}}
\newcommand{\eeqn}{\end{equation}}

\theoremstyle{plain}
\newtheorem{thm}{Theorem}

\newtheorem{cor}[thm]{Corollary}

\theoremstyle{definition}

\newtheorem{question}[thm]{Question}

\theoremstyle{remark}
\newtheorem{rem}[thm]{Remark}

\dajAUTHORdetails{%
  title = {Sumsets as Unions of Sumsets of Subsets}, 
  author = {Jordan S. Ellenberg},
  plaintextauthor = {Jordan S. Ellenberg},
  plaintexttitle = {Sumsets as Unions of Sumsets of Subsets}, 
}   

\dajEDITORdetails{%
   year={2017},
   number={14},
   received={22 May 2017},   
   published={6 September 2017},  
   doi={10.19086/da.2103},       
}   

\begin{document}

\begin{frontmatter}[classification=text]

\title{Sumsets as Unions of Sumsets of Subsets} 

\author[jse]{Jordan S. Ellenberg}

\begin{abstract}
We show that, for any subsets $S$ and $T$ of $\F_q^n$, there are subsets $S' \subset S$ and $T'  \subset T$ such that $|S'|+|T'| < c_q^n$ for some $c_q < q$, and $(S' + T) \cup (S+T') = S+T$.
\end{abstract}
\end{frontmatter}

 
The novel approach to additive combinatorics in abelian groups introduced by Croot, Lev, and Pach in \cite{CLP}  has led to rapid progress in a range of problems in extremal combinatorics:  for instance, a new upper bound for the cap set probem~\cite{ellegijs}, bounds for complexity of matrix-multiplication methods based on elementary abelian groups~\cite{blasiak}, bounds for the Erd\H{o}s-Szemeredi sunflower conjecture~\cite{naslund}, and polynomial bounds for the arithmetic triangle removal lemma~\cite{fox}.  In many of the applications, the original bound on cap sets in \cite{ellegijs} does not suffice for applications:  for instance, in \cite{blasiak} and \cite{fox} one needs to bound the size of a {\em multi-colored sum-free set}, a somewhat more general object.  

In the present note, we use the Croot-Lev-Pach lemma, combined with an older result of Meshulam on linear spaces of low-rank matrices, to prove a still more general lemma on sumsets which implies many of the combinatorial bounds used in applications so far.  Loosely speaking, we show that the sumset $S+T$ of two large subsets $S$ and $T$ of $\F_q^n$ can be expressed ``more efficiently" as a union of sumsets of smaller subsets.

We first introduce some notation.  Write $m_d$ for the number of monomials in $x_1, \ldots, x_n$ with degree at most $(q-1)$ in each variable and total degree at most $(q-1)n/3$, and write $M(\F_q^n)$ for the upper bound proved in \cite{ellegijs} for the size of a subset of $\F_q^n$ with no three-term arithmetic progressions; to be precise, we have 
\beq
M(\F_q^n) = 3m_{(q-1)n/3}
\eeq
and $M(\F_q^n)$ is bounded above by $c^n$ for some $c < q$.  (We note that for the sake of the present argument there is no need to consider prime powers $q$ other than primes.)

\begin{thm} Let $\F_q$ be a finite field and let $S,T$ be subsets of $\F_q^n$.  Then there is a subset $S'$ of $S$ and a subset $T'$ of $T$ such that
\begin{itemize}
\item $|S'| + |T'| \leq M(\F_q^n)$;
\item $(S'+T) \cup (S+T') = S+T$.
\end{itemize}
\label{th:main}
\end{thm}

Applying Theorem~\ref{th:main} to the symmetric case $S=T$, we obtain the following corollary:

\begin{cor}  Let $S$ be a subset of $\F_q^n$.  Then $S$ has a subset $S'$ of size at most $M(\F_q^n)$ such that $S'+S = S+S$.
\label{co:sym}
\end{cor}

\begin{proof} By Theorem~\ref{th:main} there are subsets $S_1$ and $S_2$ of $S$ such that $S+S = (S_1+S) \cup (S+S_2)$ and $|S_1| + |S_2| \leq M(\F_q^n)$.  Taking $S'$ to be $S_1 \cup S_2$ we are done.
\end{proof}

This immediately implies the bound proved in \cite{ellegijs} on subsets of $\F_q^n$ with no three terms in arithmetic progression:

\begin{cor}[\cite{ellegijs}]A subset $S$ of $\F_q^n$ containing no three-term arithmetic progression has size at most $M(\F_q^n)$.
\end{cor}

\begin{proof}
If $S$ has no 3-term arithmetic progression, then $S' +S$ is strictly smaller than $S+S$ for every {\em proper} subset $S' \subset S$ (because $S' + S$ fails to contain $2s$ if $s$ lies in the complement of $S'$.)  Thus, the subset $S'$ guaranteed by Corollary~\ref{co:sym} must be equal to $S$, whence $|S| = |S'| \leq M(\F_q^n)$.
\end{proof}

Theorem~\ref{th:main} also implies the bounds on multi-colored sum-free sets proved in \cite{kleinberg} and \cite{blasiak}.  (We note that \cite{blasiak} proves a substantially more general result which applies, for example, to arbitrary abelian groups of bounded exponent.)
 
\begin{cor}[Th 1, \cite{kleinberg}] Let $S,T$ be subsets of $\F_q^n$ of the same cardinality $N$, assigned an ordering $s_1, \ldots s_N$ and $t_1, \ldots, t_N$ such that the equation $s_i + t_i = s_j + t_k$ holds only when $(j,k) = (i,i)$.  Then $N \leq M(\F_q^n)$.
\label{co:multi}
\end{cor}

\begin{proof}  Let $S',T'$ be chosen as in Theorem~\ref{th:main}.  Each sum $s_i + t_i$ therefore lies in either $S+T'$ or $S'+T$.  But since $s_i + t_i$ cannot be expressed as $s_j + t_k$ for any other $j,k$, this implies that either $s_i \in S'$ or $t_i \in T'$.  It follows that $N \leq |S'| + |T'| \leq M(\F_q^n)$.
\end{proof}

We now prove Theorem~\ref{th:main}.  The proof is along the same lines as the arguments in the papers cited, but there is one new ingredient:  a result of Meshulam~\cite{meshulam:rank} on linear spaces of matrices of low rank.

\begin{proof}  Let $V$ be the space of polynomials in $\F_q[x_1, \ldots, x_n]$ with degree at most $q-1$ in each variable and total degree at most $d$, that vanish on the complement of $S+T$.  Then $\dim V$ is at least $m_d-q^n + |S+T|$.  Write $\MM$ for the space of of $|S| \times |T|$ matrices, where the rows are understood to be indexed by $S$ and the columns by $T$.

For each $P \in V$ we may consider the matrix $M(P) \in \MM$ whose entries are $P(s+t)_{s \in S, t \in T}$.  By the argument of the Croot-Lev-Pach lemma~\cite{CLP}, this matrix has rank at most $2m_{d/2}$. 

Note that $M$ is a homomorphism from $V$ to $\MM$, which is injective: if $P$ lies in the kernel, it vanishes at $S+T$, but $P$ vanishes on the complement of $S+T$, so $P$ vanishes on every point of $\F_q^n$ and is $0$.

We thus can, and shall, think of $V$ as a vector subspace of $\MM$ of dimension at least $m_d-q^n + |S+T|$, each of whose members has rank at most $2m_{d/2}$.

The arguments of \cite{ellegijs},\cite{kleinberg},\cite{blasiak} proceed by showing that, if $S,T$ satisfy the conditions Corollary~\ref{co:multi}, then $V$ contains a {\em diagonal} matrix with at least $m_d - q^n + |S|$ nonzero entries, which implies
\beq
m_d - q^n + |S| \leq 2m_{d/2},
\eeq
an upper bound on $|S|$.  The mild novelty of the present paper is to exploit the Croot-Lev-Pach rank bound for the whole space $V$, not just for its subspace of diagonal matrices.  The earlier papers use the easy fact that a vector space of diagonal matrices of dimension at least $r$ contains a matrix of rank at least $r$.  For spaces of general matrices, the problem of controlling the maximal rank attained in a linear space of matrices is much richer.  We will use a theorem of Meshulam~\cite[Theorem 1]{meshulam:rank} in this area, which (rather surprisingly to us) turns out to be perfectly adapted to the combinatorial application.  (Indeed, we did not set out to prove Theorem~\ref{th:main}; rather, we encountered Meshulam's theorem and simply worked out what it had to say about sumsets when combined with the argument of \cite{ellegijs}.)
 

In the interest of self-containedness, we state Meshulam's theorem below.  

\begin{thm}[Meshulam]
Let $k$ be a field and $W$ a vector subspace of $M_n(k)$.  For each $w \in W$ let $p(w) \in \set{1,\ldots,n} \times \set{1,\ldots,n}$ be the lexicographically first $(i,j)$ such that the entry $w_{ij}$ is nonzero, and let $\Sigma$ be the set of all $p(w)$ as $w$ ranges over $W$.  Suppose every matrix in $W$ has rank at most $r$.  Then there exists a set of $m$ rows and $m'$ columns such that  every element of $\Sigma$ is contained in one of the rows or one of the columns, and $m+m' \leq r$.
\label{meshulam}
\end{thm}

We now return to the proof of Theorem~\ref{th:main}.  Choose an ordering on $S$ and an ordering on $T$.  These choices endow the entries of a matrix in $\MM$ with a lexicographic order.  As above, for each matrix $A \in \MM$,  we denote by $p(A) \in S \times T$ the location of the lexicographically first nonzero entry of $A$.

We note that $p(M(P))$ cannot be an arbitrary element of $S \times T$, since $M(P)$ has equal entries at $(s,t)$ and $(s',t')$ whenever $s+t = s'+t'$.  In particular, this means that $(s,t)$ and $(s',t')$ cannot both be $p(M(P))$ for polynomials $P \in V$; only the lexicographically prior of these two pairs can appear.

By Gaussian elimination, there is a basis $A_1, \ldots, A_{\dim V}$ for $V$ such that $p(A_1), \ldots, p(A_{\dim V})$ are distinct.  Now apply Theorem~\ref{meshulam}, which shows that there is a set of $2m_{d/2}$ lines (a line being a row or a column) whose union contains $p(A_i)$ for all $i$.

This set of lines consists of a subset of $S$, which we call $S_0$, and a subset of $T$, which we call $T_0$, satisfying $|S_0|+ |T_0| = 2m_{d/2}$.  

We now have, for $i=1,\ldots,\dim V$,
\beq
p(A_i) = (s_i,t_i)
\eeq
with either $s_i \in S_0$ or $t_i \in T_0$.  What's more, $s_i + t_i$ and $s_j + t_j$ are distinct whenever $i$ and $j$ are.  So the union of $S_0 + T$ with $S + T_0$ contains at least $\dim V$ elements of $S+T$.

Since $\dim V \geq m_d - q_n +  |S+T|$, the set $W$ of elements of $S+T$ {\em not} contained in $(S_0 + T) \cup (S + T_0)$ has cardinality at most $q_n - m_d$.  Let $S_1$ be a subset of $S$ of size $q_n - m_d$ such that each $w \in W$ is represented as $s+t$ for some $s \in S_1$.  Then taking $S' = S_0 \cup S_1$ and $T' = T_0$, we have that $S' + T \cup S + T' $ contains all of $S+T$; moreover,
\beq
|S'| + |T'| \leq 2m_{d/2} + q^n - m_d
\eeq
and minimizing over $d$ we get the desired result.
\end{proof}


\begin{rem} We note that the algebraic approach to bounding sumsets is much older than \cite{CLP} and \cite{ellegijs}; one ancestor, for instance, is Alon's short proof of the Erd\H{o}s-Heilbronn conjecture via combinatorial Nullstellensatz~\cite[Prop 4.2]{alon}, which also proceeds by considering algebraic properties of a polynomial vanishing on the set of distinct sums in an abelian group (in that case a cyclic group.)
\end{rem}

\begin{question}  Corollary~\ref{co:multi}, the bound on multi-colored sum-free sets, can be expressed in a more symmetric, and thus more appealing, form:  Suppose $S,T,U$ are subsets of $\F_q^n$ such that the set
\beq
\set{(s,t,u) \in S \times T \times U: s+t+u=0}
\eeq
forms a perfect matching between the three sets.  Then $|S| = |T| = |U|$ is at most $M(\F_q^3)$.  The proof, too, has a symmetric formulation; Tao introduced the notion of {\em slice rank} for tensors in  $\F_q^n \tensor \F_q^n \tensor \F_q^n$, which was quickly generalized in many directions and applied to a range of further combinatorial problems (see e.g. \cite{naslund:multi}.)

Symmetric methods of this type seem to be the most elegant way to approach these problems. Is there a way to state Theorem~\ref{th:main}, and prove it, as a statement about solutions to $s+t+u = 0$ which places the three summands on an equal footing?
\end{question}

\begin{question}  One naturally wonders whether Theorem~\ref{th:main} has an analogue for cyclic groups.  That is:  let $g(N)$ be the smallest integer such that, for any subsets $S$ and $T$ of $\Z/N\Z$, there are always $S' \subset S$ and $T' \subset T$ with $(S+T') \cup (S'+T) = S+T$ and $|S'| + |T'| \leq g(N)$.  What can we say about the growth of $g(N)$?   Behrend's example~\cite{behrend} of a large subset of $\Z/N\Z$ with no three-term arithmetic progressions shows that $g(N)$ would have to be at least $N^{1-\epsilon}$.  Jacob Fox and Will Sawin explained to me that $g(N) = o(N)$ follows from known bounds for arithmetic triangle removal. 
\end{question}





\section*{Acknowledgments} 
The author is supported by NSF Grant DMS-1402620 and a Guggenheim Fellowship.  He thanks Jacob Fox, Will Sawin, the referees, and the readers of Quomodocumque for useful discussions about the subject of this paper.


%
%

\begin{dajauthors}
\begin{authorinfo}[pgom]

Jordan S. Ellenberg \\
Professor \\
University of Wisconsin-Madison \\
ellenber\imageat{}math\imagedot{}wisc\imagedot{}edu \\
\url{http://www.math.wisc.edu/~ellenber}
%
%
\end{authorinfo}
\end{dajauthors}

\end{document}